\documentclass[11pt,reqno]{amsart}

\usepackage[T1]{fontenc}
\usepackage[english]{babel}
\usepackage{amsmath,amssymb,amscd,amsfonts,mathtools,tikz,caption,float,tikz-cd}
\usetikzlibrary{patterns}
\usepackage[linkcolor=blue,colorlinks=true,citecolor=blue]{hyperref}
\usepackage[capitalise]{cleveref}

\newtheorem{thm}{Theorem}[section]
\newtheorem{lm}[thm]{Lemma}

\newtheorem{rmk}[thm]{Remark}

\newcommand{\figref}[1]{\hyperref[#1]{Figure \ref{#1}}}
\newcommand{\lemref}[1]{\hyperref[#1]{Lemma \ref{#1}}}
\newcommand{\thmref}[1]{\hyperref[#1]{Theorem \ref{#1}}}
\newcommand{\conjref}[1]{\hyperref[#1]{Conjecture \ref{#1}}}
\newcommand{\propref}[1]{\hyperref[#1]{Proposition \ref{#1}}}
\newcommand{\corref}[1]{\hyperref[#1]{Corollary \ref{#1}}}
\newcommand{\defref}[1]{\hyperref[#1]{Definition \ref{#1}}}
\newcommand{\rmkref}[1]{\hyperref[#1]{Remark \ref{#1}}}
\newcommand{\qref}[1]{\hyperref[#1]{Question \ref{#1}}}
\newcommand{\secref}[1]{\hyperref[#1]{\S\ref{#1}}}
\newcommand{\appref}[1]{\hyperref[#1]{Appendix \ref{#1}}}

\newcommand{\R}{\mathbf{R}}

\newcommand{\h}{\mathbf{H}}
\newcommand{\Q}{\mathbf{Q}}
\newcommand{\Z}{\mathbf{Z}}

\newcommand{\SL}{\mathrm{SL}}

\newcommand{\PSL}{\mathrm{PSL}}

\newcommand{\cC}{\mathcal{C}}

\newcommand{\cN}{\mathcal{N}}

\newcommand{\ga}{\alpha}     %lowercase alpha
\newcommand{\gb}{\beta}      %lowercase beta
\newcommand{\G}{\Gamma}      %Capital Gamma
\newcommand{\g}{\gamma}      %lowercase gamma
\newcommand{\bk}{\backslash}
\newcommand{\cro}{{\rm cr}}
\newcommand{\tr}{\operatorname{tr}}
\newcommand{\Nr}{\operatorname{Nr}}

\newcommand{\bbm}{\begin{bmatrix}}
\newcommand{\ebm}{\end{bmatrix}}
\newcommand{\bpm}{\begin{pmatrix}}
\newcommand{\epm}{\end{pmatrix}}
\newcommand{\bsm}{\left(\begin{smallmatrix}}
\newcommand{\esm}{\end{smallmatrix}\right)}
\newcommand{\bsbm}{\left[\begin{smallmatrix}}
\newcommand{\esbm}{\end{smallmatrix}\right]}

\numberwithin{equation}{section}

\title{Modular systoles are extremal for the crossing number}

\author{Claire Burrin}
\address{Institute of Mathematics, University of Zurich, Switzerland}
\email{claire.burrin@math.uzh.ch}
\author{Hugo Parlier}
\address{Department of Mathematics, University of Fribourg, Switzerland}
\email{hugo.parlier@unifr.ch}
\thanks{C.B. acknowledges the support of SNSF grant number 201557}
\begin{document}

\maketitle

\begin{abstract}
    We study crossing numbers for systoles of congruence surfaces. Taken as a family of curves on a family of surfaces, we show that the growth rate of their intersection is optimally small among all sets of curves of the same cardinality lying on the same topological surface. 
\end{abstract}

\section{Introduction}
Crossing numbers are well-studied quantities in the context of graph theory that serve as a measure of non-planarity. By analogy with crossing lemmas for graphs, recent results have explored to which extent collections of curves on surfaces must cross in terms of the number of curves and the topology of the surface. In this context, we analyze an important family of curves on surfaces which arise on the interface of hyperbolic geometry and number theory. 

Let $X(N)=\G(N)\bk \h$ be the congruence surface of level $N$, where 
\begin{align*}
    \G(N) &=\ker\{\PSL_2(\Z)\to\PSL_2(\Z/N\Z)\}\\
    & = \left\{ A \in \SL_2(\Z) \mid A\equiv \pm I \, (\text{mod }N) \right\}/\pm I
\end{align*}
is the principal congruence subgroup of level $N$. These surfaces have suggestively been described as "hedgehog shaped"; they are noncompact surfaces for which the genus and number of cusps grow with order $N^3$ and $N^2$ respectively. 

A further descriptive invariant of surfaces is given by its systoles (the family of shortest closed geodesics). In their moduli space, the surfaces $X(N)$ have extremal systole length $\mathrm{sys}(X(N))$ (see \cite{SS94} and \cite{Bavard1996}) and it is believed that their number is also maximal. It was recently established that they are also asymptotically minimal with respect to the crossing number; namely if $\cC(N)$ denotes the family of systoles on $X(N)$, then
\begin{align}\label{lowerbound}
\underset{N\to\infty}{\lim\sup}\, \frac{\log\cro(\cC_N)}{\log N} \geq 5,    
\end{align}
for any family $\cC_N$ of curves on $X(N)$ with $|\cC_N|\approx |\cC(N)|$ and that this lower bound is sharp precisely for $\cC(N)$ \cite{BBS}.The lower bound (\ref{lowerbound}) is established based on a `soft' argument, relying essentially on the growth of the Euler characteristic. A stronger lower bound found in \cite{HP} implies that for a countably infinite family $\cN$ of $N$'s we have the arithmetic lower bound
\begin{align}\label{HP}
\cro(\cC(N)) \gg N^5 L(1,\chi)^2,
\end{align}
where $\gg$ is the Vinogradov symbol\footnote{The Vinogradov symbol $\ll$ means that $A\leq cB$ for some constant $c>0$ that may depend on the context but not on the quantities $A$ and $B$. The notation $A\asymp B$ means that $A\ll B\ll A$.} and $L(1,\chi)$ is the special value at $1$ of the Dirichlet $L$-function associated to the Kronecker symbol $\chi(n)=(\tfrac{N^2-4}{n})$. (Although $\chi$ and the associated special $L$-value depend on $N$, we drop the reference to $N$ to avoid overly cumbersome notation.) Determining the size of $L(1,\chi)$ is a central topic in analytic number theory: we have the established known bounds
\begin{align*}
    N^{- \epsilon} \ll_\epsilon  L(1,\chi) \ll \log N,
\end{align*}
whereby the upper bound is elementary and the lower bound follows from deep work of Siegel, while under the additional assumption of GRH we have the stronger bounds
\begin{align*}
    (\log\log N)^{-1} \ll L(1,\chi) \ll \log\log N.
\end{align*}

The family $\cN$ for which (\ref{HP}) holds is explicit: 
$$
\cN \coloneqq  \{ N\geq 3 \mid N^2-4 \text{ is squarefree} \}.
$$
It is an infinite family, and in fact it has positive density in the natural numbers. On the other hand a more careful reading of the argument in \cite{BBS} yields the upper bound
\begin{align}\label{BBS}
    \cro(\cC(N)) \ll N^5 L(1,\chi)^2
\end{align}
for all $N\in\cN$. Putting these results together, we obtain the following growth estimate. 
\begin{thm}
For $N\in \cN$, we have 
$
\cro(\cC(N)) \asymp N^{5} L(1,\chi)^2.
$
\end{thm}

This shows that $\cC(N)$ is extremal for the crossing number in a stronger sense. More precisely, the crossing lemma in \cite{HP} gives an inequality on the crossing number of a family of curves which relates the Euler characteristic and the number of curves in the family. The inequality really relies on these two variables and has been shown to be sharp in terms of the growth of the number of curves (for fixed Euler characteristic). The result above shows that the inequality is also sharp in a hybrid regime where both the topology and the number of curves grow in function of $N$. More precisely, if $\Gamma$ is any set of curves of the same cardinality as $\cC(N)$ on the same surface, then $\cro(\Gamma) \gg \cro(\cC(N))$.

The purpose of this note is both to record this result and to present a relatively self-contained proof of it, which we now sketch. For the lower bound (\ref{HP}), we rely on the crossing lemma in \cite{HP}, but in fact, one could almost mimic the proof directly. The crossing lemma 
gives an explicit lower bound on the crossing of a curve system, and its proof relies on encoding the intersections of the curve system (in our case $\cC(N)$) as a graph drawing on a topological surface $\Sigma$. An application of the circle packing theorem for triangulations embedded on surfaces, a Cauchy-Schwartz argument and a quantitative version of the prime geodesic theorem then yields the lower bound. To prove the upper bound (\ref{BBS}), we also rely on the work of Schmutz Schaller (as extended in \cite{BBS}), which gives us
\begin{align*}
    \cro(\cC(N))\ll N^3 \times \# \text{ intersecting closed geodesics of trace }N \text{ on }X(1).
\end{align*}
If on the right hand-side we were restricting to \emph{primitive} closed geodesics on $X(1)$, then (\ref{BBS}) would directly follow from the work of Jung and Sardari \cite{JS} on the distribution of intersecting modular geodesics and  Dirichlet's class number formula. To access their result, we show using the prime geodesic theorem that the contribution of nonprimitive closed geodesics is marginal.

\section{Background: Modular Geodesics}

In his approach to showing that modular surfaces are extremal for systole length, Schmutz Schaller proved that the isometry classes of systoles on $X(N)$ correspond to the closed (not necessarily primitive) geodesics of trace $N$ on $X(1)=\G(1)\backslash\h=\PSL(2,\Z)\backslash\h$. 

We call a primitive oriented closed geodesic on $X(1)$ a modular geodesic. To study modular geodesics, we have access to a completely number theoretic toolbox thanks to the following (classical) correspondence.
\begin{thm}
Modular geodesics  are in one-to-one correspondence with primitive integral binary quadratic forms (PIBQF) of positive (squarefree) discriminant.
\end{thm} 
\begin{proof}
The correspondence is as follows: given a PIBQF $Q(x,y)=Ax^2+Bxy+Cy^2$, the quadratic equation $Q(x,1)=Ax^2+Bx+C=0$ has discriminant $D=B^2-4AC$ and two (quadratic irrational) solutions $x_1,x_2\in\R$, which can be seen as the endpoints of an oriented geodesic $\g_Q$ in the hyperbolic upper half-plane $\h$. The stabilizer group of this geodesic under the action of ${\rm Isom}^+(\h)$ is a cyclic group with generator a hyperbolic isometry. If we start from the matrix representative $A=\bsm a & b\\ c &d\esm\in \SL_2(\Z)$ that respects the orientation of $\g_Q$ and take the associated quadratic form $Q_A(x,y)=bx^2+(d-a)xy - cy^2$, we find that $Q_A(x,1)=0$ has the same discriminant $D$ and solutions $x_1,x_2$. Next, we consider the dual action of $\SL_2(\Z)$ on the domain of the quadratic form by linear change of variables and on $\h$ by fractional linear transformation. These actions quotient through $\PSL_2(\Z)$ and are preserved by the correspondence described here. We conclude with the observation that all isometric images of $\g_Q$ under $\PSL_2(\Z)$ project to a single primitive closed oriented geodesic on $X(1)$.
\end{proof}
\begin{rmk}
The proof shows that modular geodesics of trace $N$ have discriminant $D=N^2-4.$
\end{rmk}

The number of PIBQF of discriminant $D$ is finite, given by the class number $h(D)$, and it has been known since Gauss that $1\leq h(D)<\infty$. (His famous conjecture that $h(D)$ is infinitely often $1$ when $D$ is a fundamental discriminant is still open.) Mostly any information available on $h(D)$ stems from Dirichlet's class number formula, 
\begin{align*}
h(D)\log \epsilon_D = \sqrt D L(1,\chi),
\end{align*}
which holds whenever $D$ is a fundamental discriminant, i.e., when $D\equiv 1$ (mod 4) and is squarefree or if $D$ is a multiple of 4 and $D/4$ is squarefree with $D/4\equiv 2,3$ (mod 4) --- that is, whenever $D$ is the discriminant of the quadratic field $\Q(\sqrt D)$.

Every one of the finitely many modular geodesics of (fundamental) discriminant $D$ has length $\log \epsilon_D$, with $\epsilon_D=\tfrac{s+t\sqrt D}{2}$ where $(s,t)$ is the minimal solution to the Pell equation $s^2-t^2 D=4$. 

\begin{lm}\label{lm:positive density}
There are infinitely many $N\geq 3$ such that $D=N^2-4$ is a fundamental discriminant.
\end{lm}
\begin{proof}
It suffices to know that every squarefree discriminant $D\equiv1$ (mod 4) is a fundamental discriminant. An old result of Estermann states that there is a positive density of natural numbers $N$ such that $N^2-4$ is squarefree. Since this also forces $N$ to be odd, the corresponding discriminants $D=N^2-4$ are fundamental discriminants, and there are infinitely many such $N$'s.
\end{proof}

\begin{lm}\label{lm:size fundamental unit}
For $D=N^2-4>0$ squarefree, we have $\epsilon_D = \tfrac{N+\sqrt{N^2-4}}{2}$.
\end{lm}
\begin{proof}
Let $\epsilon_0 = \tfrac{s_0+t_0\sqrt D}{2}$ be the fundamental unit of the quadratic field $\Q(\sqrt D)$; that is, $(s_0,t_0)$ is the minimal solution to the Pell equation $s^2-t^2 D=4$.
For any further solution $(s,t)$, $\epsilon = \tfrac{s+t\sqrt D}{4}$ is a power of $\epsilon_0$. In particular, $\epsilon_0^2 \leq \epsilon$.

Suppose that $\epsilon_0\neq \epsilon_D$ as in the statement, then
\begin{align*}
   4t_0^2 D <  s_0^2 + t_0^2 D + 2s_0 t_0 \sqrt D = (s_0+ t_0 \sqrt D)^2 \leq 2(N+\sqrt D)
\end{align*}   
where the inequality on the left hand-side is obtained using $s_0^2 = 4+ t_0^2 D > t_0^2 D$ and the inequality on the right hand-side is obtained using $\epsilon_0^2\leq \epsilon_D$. But this implies $N^2-4 = D  \leq t_0^2 D < \tfrac{N+\sqrt D}{2}<N$, which is false for $N\geq 3$.
\end{proof}

\section{Lower Bound}

Set $D=N^2-4$. To access Dirichlet's class number formula, we will restrict to $N$ such that $D$ runs over fundamental discriminants. Let $\cC(N)$ be the family of systoles on $X(N)$ and $|\cC(N)|$ be the number of systoles. Schmutz Schaller \cite{SS} proved that closed (not necessarily primitive) geodesics on $X(1)$ of discriminant $D$ are in one-to-one correspondence with isometry classes of systoles on $X(N)$ so that 
\begin{align*}
    |\cC(N)|\gg |\G(1)/\G(N)|\, h(D).
\end{align*}
The subgroup $\G(N)$ is normal in $\G(1)$ with quotient isomorphic to the finite group $\PSL_2(\Z/N\Z)$ of order $\tfrac{1}{2}N^3\prod_{p\mid N}(1-p^{-2})\asymp N^3$. From the class number formula, $h(D)\log \epsilon_D = \sqrt D L(1,\chi)$, available since $D$ is a fundamental discriminant, and using that $\epsilon_D \asymp \sqrt D$ (\cref{lm:size fundamental unit}), we have here the lower bound
$$
|\cC(N)| \gg \frac{N^4}{\log N} L(1,\chi).
$$ 

We now introduce the crossing lemma result that will give us the lower bound. In its statement, $\chi=\chi(\Sigma)$ is understood to denote the Euler characteristic of $\Sigma$ and not the Kronecker symbol. Since both notations are standard in their respective fields, we have decided to keep them as such; it should be clear from context to which we refer.
\begin{thm}[Thm 1.2 \cite{HP}]\label{thm:HP}
Let $\G$ be a collection of $m$ distinct homotopy classes of non-trivial closed and primitive curves on a surface $\Sigma$ of Euler characteristic $\chi$. Then, for all $m\geq e^6(|\chi|+1)$, we have
\begin{align*}
    \cro(\G) > \frac{1}{128|\chi|}\left( m\log\left(\frac{m}{(|\chi|+1)e^6}\right)\right)^2.
\end{align*}
\end{thm}
The genus of $X(N)$ is roughly $N^3$ and it has roughly $N^2$ cusps. Ultimately, what we require is a topological configuration of curves on a surface. For simplicity, we consider the cusps as boundary curves and add one-holed tori to each of them to obtain a closed surface. This increases the genus by at most roughly $N^2$, so in particular the genus is still on the order of $N^3$. The number of curves in our set is unchanged and so $m=|\cC(N)|$ and $|\chi|\asymp N^3$. The above bound thus gives
\begin{align*}
    \cro(\cC(N)) \gg \frac{|\cC(N)|^2 \log(|\cC(N)|/|\chi|)^2}{|\chi|} \gg N^5 L(1,\chi)^2 
\end{align*}
as required. 

It is interesting to note that Theorem \ref{thm:HP} is shown to have  optimal growth rate, for fixed topology, by considering large sets of curves that are non-simple. And in fact, for simple curves the result is not even roughly optimal in the fixed topology regime. In our context however, via the upper bound we proceed to show, our optimal sets of curves are simple. It would be interesting to know exactly under which regimes simple curves fail to reach the optimal bounds. We note that in genus $2$, actual optimal results are known for small curve systems \cite{Jo25}.
\section{Upper Bound}
We again restrict to $N\geq 3$ such that $D=N^2-4$ is squarefree. Let $I(N)$ be the number of intersecting closed (not necessarily primitive) geodesics on $X(1)$ of trace $N$, or equivalently of discriminant $D=N^2-4$. In \cite{BBS} we extended the Schmutz Schaller correspondence \cite{SS} to show that 
\begin{align*}
    \cro(\cC(N)) \ll N^3 I(N).
\end{align*}
It turns out that the contribution to $I(N)$ of nonprimitive geodesics is marginal (as is the case in the prime geodesic theorem) so that $I(N)$ can be estimated directly using the distribution results of Jung and Sardari on intersecting geodesics on the modular surface (in particular \cite[Thm 1.4]{JS}) and the class number formula.
We carry this argument carefully below to obtain a more precise bound on the crossing number than what had been done in \cite{BBS}.
\begin{thm}\label{main}
For $N\geq 3$ such that $D=N^2-4$ is squarefree, we have
\begin{align*}
    I(N) \ll N^2 L(1,\chi)^2,
\end{align*}
for the Kronecker symbol $\chi(n)=(\tfrac{N^2-4}{n})$.
\end{thm}

\begin{proof}
We start by expressing $I(N)$ as a multiple sum over modular geodesics on $X(1)$ as follows,
\begin{align*}
      I(N)& = \sum_{j,k\geq1} \sum_{\substack{\g,\g' \text{ prim}\\ \tr(\g^j),\tr(\g'^{k})=N}} 1_{\g^j \pitchfork \g'^k},
\end{align*}
where $1_{\ga \pitchfork \gb}=1$ if the geodesics $\ga$ and $\gb$ intersect and $0$ otherwise. Recall that $\tr(\g)=\Nr(\g)^{1/2}+\Nr(\g)^{-1/2}$, where $\Nr(\g)=e^{\ell(\g)}$ is the `norm' of the closed geodesic $\g$, which is multiplicative. From the rough estimate $\tr(\g)-1<\Nr(\g)^{1/2}<\tr(\g)$, we have that if $\tr(\g^j)=N$ then 
\begin{align*}
    N-1 < \Nr (\g)^{j/2} < N.
\end{align*}
This implies in particular that $j\ll \log N$. We also note that $1_{\g^k\pitchfork \g'^k}=1_{\g\pitchfork \g'}$ so that we may rewrite 
\begin{align*}
  I(N) =\sum_{1\leq j,k\ll N} \sum_{\substack{\g,\g' \text{ prim}\\ \tr(\g^j),\tr(\g'^{k})=N}} 1_{\g \pitchfork \g'}.
\end{align*}
We denote the discriminant of a primitive $\g$ for which $\tr(\g^j)=N$ by $D_j$. Note that if $\g$ is a modular geodesic of discriminant $D_j$ then $\tr(\g^j)=N$. Now \cite[Theorem 1.4]{JS} implies that 
\begin{align*}
   \sum_{\substack{\g,\g' \text{ prim}\\ D(\g)=D_j,D(\g')=D_k}} 1_{\g \pitchfork \g'} \ll h(D_j)\log \epsilon_{D_j} h(D_k)\log \epsilon_{D_k}
\end{align*}
Hence
\begin{align*}
   I(N)
    &\ll \left(\sum_{1\leq j\ll \log N} h(D_j)\log \epsilon_{D_j}\right)^2.
\end{align*}
When $j=1$, $D_j=D$ is a fundamental discriminant and we can apply the class number formula to obtain 
$$
h(D)\log \epsilon_D  \ll N L(1,\chi).
$$
For $j\geq 2$, we cannot guarantee that $D_j$ is a fundamental discriminant, and so we rely on the prime geodesic theorem 
\begin{align*}
   \Pi(x)\coloneqq  \sum_{\substack{\g \text{ prim}\\ \Nr(\g)< x}} \log \Nr(\g) & = x + O(x^{3/4}).
\end{align*}
(Stronger bounds on the error term are known, starting with, e.g., Soundararajan--Young, but these are not needed here.) Then 
\begin{align*}
    h(D_j)\log \epsilon _{D_j} &= \sum_{\substack{\g \text{ prim}\\ D(\g)=D_j}} \log \Nr(\g)  \leq \Pi(N^{2/j})-\Pi((N-1)^{2/j}) = O(N^{3/2j}).
\end{align*}
since $N^{2/j}-(N-1)^{2/j} \ll N^{2/j-1} \ll N^{3/2j}$ when $j\geq2$. We thus have
\begin{align*}
    I(N) \ll \left( NL(1,\chi) +  O(N^{3/4}\log N)\right)^2 = N^2L(1,\chi)^2 + O(N^{7/4}(\log N)^2).
\end{align*}
Siegel showed that $L(1,\chi)\gg_\epsilon  N^{-\epsilon}$ so that we can conclude that $I(N)\ll N^2 L(1,\chi)^2$.
\end{proof}

\end{document}